\documentclass{birkjour} 
\usepackage{amsmath}
\usepackage{amssymb}
\usepackage{mathabx}
\usepackage{xcolor}

\newtheorem{theorem}{Theorem}[section]

\newtheorem{lemma}[theorem]{Lemma}
\newtheorem{corollary}[theorem]{Corollary}
\newtheorem{definition}[theorem]{Definition}
\theoremstyle{definition}
\newtheorem{remark}[theorem]{Remark}
\numberwithin{equation}{section}

\newcommand{\krein}{Kre{\u\i}n }
\newcommand{\ip}[2]{\left<#1,#2\right>}
\newcommand{\cA}{{\mathcal A}}
\newcommand{\cB}{{\mathcal B}}
\newcommand{\cC}{{\mathcal C}}

\newcommand{\cG}{{\mathcal G}}
\newcommand{\cH}{{\mathcal H}}
\newcommand{\cK}{{\mathcal K}}
\newcommand{\cL}{{\mathcal L}}
\newcommand{\cM}{{\mathcal M}}
\newcommand{\cN}{{\mathcal N}}

\newcommand{\cX}{{\mathcal X}}
\newcommand{\fdh}{\cH = \cH_+ \oplus \cH_-}

\newcommand{\lh}{\cL(\cH)}
\newcommand{\lhk}{\cL(\cH,\cK)}
\newcommand{\lkh}{\cL(\cK,\cH)}
\newcommand{\lk}{\cL(\cK)}
\newcommand{\ran}{{\rm ran\,}}
\renewcommand{\ker}{{\rm ker\,}}

\newcommand{\indplusminus}{{\rm{ind}_\pm}\,}

\newcommand{\lah}{\cL(\cA,\cH)}
\newcommand{\lak}{\cL(\cA,\cK)}
\newcommand{\lbh}{\cL(\cB,\cH)}

\begin{document}

\title[Factorization of selfadjoint operators]{Hermitian indices and
  factorization of selfadjoint operators on a Kre{\u\i}n space}

\author[M.\ A.\ Dritschel]{Michael A.\ Dritschel}
\address{School of Mathematics, Statistics, \& Physics, Newcastle
  University, Newcastle upon Tyne, NE1 7RU, United Kingdom}
\email{michael.dritschel@newcastle.ac.uk} 

\author[A.\ Maestripieri]{Alejandra Maestripieri}
 \address{Instituto Argentino de Matem\'atica ``Alberto
   P. Calder\'on'' CONICET, Saavedra 15, Piso 3, (1083) Buenos Aires,
   Argentina}
\email{amaestripieri@gmail.com}

\author[J.\ Rovnyak]{James Rovnyak}
\address{Department of Mathematics, University of Virginia,
P.O.\ Box 400137, Charlottesville, Virginia, U.S.A.}
\email{rovnyak@virginia.edu}

 \subjclass{46C20, 46B50}
 
 \dedicatory{ \textcolor{black}{In memory of Heinz Langer, for his
     many contributions and leadership in the study of indefinite
     inner product spaces.} }

 \keywords{\krein space, selfadjoint operator, hermitian indices, rank
   of positivity, congruence, Bogn{\'a}r-Kr{\'a}mli factorization,
   induced \krein space, continuously contained \krein space}

 \begin{abstract}
   The hermitian indices of a selfadjoint operator $C$ on a \krein
   space $\cH$ are defined as geometric measures of positivity and
   negativity of the operator.  A different pair of indices arises in
   the Bogn{\'a}r-Kr{\'a}mli factorization of~$C$, which writes $C$ as
   a product $AA^*$ where $A$ acts on a \krein space $\cA$ into $\cH$
   and has zero kernel; the new indices are the positive and negative
   indices of~$\cA$.  Such factorizations are far from unique.  When
   $\cH$ is separable, it is known that the two notions of indices
   always coincide, and this has applications to index formulas in the
   theory of Julia operators and completion problems for operator
   matrices.   A new proof of the equality of indices that does not
  require separability is given in this work.  
\end{abstract}
\maketitle

\section{Introduction}\label{S:Introduction and preliminaries}

The notion of hermitian indices $h_\pm(C)$ for a selfadjoint operator
$C$ on a \krein space $\cH$, defined below, is closely related to
representations of $C$ as the product of an operator $A$ and its
adjoint.  The existence of such representations with $A\in\lh$ was
characterized by J.\ Bogn{\'a}r and A.\ Kr{\'a}mli \cite[Th.\
1]{BognarKramli} in terms of ranks of positivity and negativity of $C$
(see also Bogn{\'a}r \cite[Th.\ VII.2.1]{Bognar1974}).  A related
factorization $C=AA^*$ with $A$ acting from an external \krein space
$\cA$ into $\cH$ and $\ker A = \{0\}$ was encountered in applications
by two of the present authors \cite{DR1990, DRFields}.  Of primary
interest in the applications is the fact that the indices
$\indplusminus \cA$ of the external \krein space $\cA$ do not depend
on the choice of factorization.  For separable \krein spaces, it is
shown in \cite{DR1990} that this is always the case, and in fact
$\indplusminus \cA$ coincide with an ad hoc definition of hermitian
indices $h_\pm(C)$ for separable \krein spaces \cite[\S1.2]{DR1990}.
In this paper we extend the ad hoc definition of hermitian indices to
arbitrary \krein spaces and show that $\indplusminus \cA = h_\pm(C)$
for nonseparable as well as separable \krein spaces.  This effectively
removes the assumption of separability in index formulas that appear
in \cite{DR1990, DRFields}.

Notation and terminology generally follow \cite{DR1990, DRFields}.  A
\krein space is a complex inner product space
$(\cH, {\ip{\cdot}{\cdot}}_{\cH})$ which admits a {\bf fundamental
  decomposition} $\fdh$ such that $\cH_+$ is a Hilbert space and
$\cH_-$ is the antispace of a Hilbert space.  The indices of~$\cH$,
$\indplusminus \cH = \dim \cH_\pm$, do not depend on the choice of
fundamental decomposition.  Every fundamental decomposition induces an
{\bf associated Hilbert space} $|\cH|$ which coincides with $\cH$ as a
vector space and replaces $\cH_-$ by the original Hilbert space.  A
corresponding {\bf norm} for the \krein space~$\cH$ is defined by
$ \|f\|^2 = {\ip{f_+}{f_+}}_{\cH} + |{\ip{f_-}{f_-}}_{\cH}| $, where
$f_\pm$ are the components of $f$ in~$\cH_\pm$.  Any two norms are
equivalent and determine a unique {\bf strong topology} on~$\cH$.  A
{\bf subspace} is simply a linear manifold and need not be closed.
The {\bf dimension} of a closed subspace is its dimension in any
associated Hilbert space (the choice does not matter).  Operators are
assumed everywhere defined and continuous.  If $\cH$ and $\cK$ are
\krein spaces, $\lh$ and $\lhk$ are the spaces of continuous operators
on $\cH$ into itself and on $\cH$ into $\cK$.  The \krein space
adjoint of an operator $A$ is denoted~$A^*$.
If $C$ is a selfadjoint operator on a \krein space~$\cH$, the formula 
\begin{equation*}
  {\ip{f}{g}}_C = {\ip{Cf}{g}}_\cH, \qquad f,g\in\cH,
\end{equation*}
defines a linear and symmetric {\bf $C$-inner product} on~$\cH$.  The
induced quadratic form ${\ip{Cf}{f}}_\cH$, $f\in\cH$, thus assumes
only real values.  A subspace $\cM$ of $\cH$ is $C$-{\bf strictly
  positive} if ${\ip{f}{f}}_C > 0$ for every $f\neq0$ in~$\cM$, and
$C$-{\bf strictly negative} if ${\ip{f}{f}}_C < 0$ for every $f\neq0$
in~$\cM$; {\bf $C$-orthogonality} for vectors in~$\cH$, $\perp_C$,
refers to orthogonality in the $C$-inner product.

\textcolor{black}{ Section \ref{Sect2} presents a self-contained
  account of hermitian indices and includes a technical result,
  Theorem \ref{keyth}, on signature operators on a Hilbert space.  The
  theorem is used in Section~\ref{Sect3} to prove our main result,
  Theorem \ref{mainth}, which is also stated in an equivalent form for
  continuously contained \krein spaces.}

\section{Hermitian indices}\label{Sect2}

\textcolor{black}{ The concepts of ranks of positivity and negativity
  for operators and quadratic forms are classical and appear already
  in elementary matrix theory.  For operators on \krein spaces, they
  are formalized in Bogn{\'a}r \cite{Bognar1974} in terms of
  decomposable inner product spaces.  Hermitian indices provide an
  equivalent viewpoint that reduces to the approach of Potapov
  \cite[Ch.\ 2]{Potapov1955} for finite-dimensional \krein spaces.}

Some preliminaries on the notion of congruence are needed.
\krein space operators $A\in\lh$ and $ B\in\lk$ are said to be {\bf
  congruent} if $A=X^*BX$ for some invertible operator $X\in\lhk$.
Congruence is an equivalence relation.
    
\begin{lemma}\label{LemmaA}
  Every selfadjoint operator on a \krein space is congruent to a
  selfadjoint operator on a Hilbert space.
\end{lemma}

\begin{proof}
  Let $\cH$ be a \krein space, $C\in\lh$ a selfadjoint operator.
  Choose any invertible operator $X$ from $\cH$ onto a Hilbert
  space~$\cK$.  Then since $X^{*\,-1} = X^{-1\,*}$, the operator
  $D=X^{*\,-1}CX^{-1}$ is selfadjoint on $\cK$ and $C=X^*DX$.
\end{proof}

\begin{lemma}\label{LemmaB}
  Let $A\in\lh$ and $ B\in\lk$ be congruent selfadjoint operators on
  \krein spaces $\cH$ and $\cK$, and suppose $A=X^*BX$ where
  $X\in\lhk$ is invertible.
  Then the mapping $$\cX\colon \cM \to X\cM$$
  is a one-to-one and onto correspondence between the set
  $\cC_\pm(\cH,A)$ of all closed $A$-strictly positive/negative
  subspaces of $\cH$ and the set $\cC_\pm(\cK,B)$ of all closed
  $B$-strictly positive/negative subspaces of $\cK$.  Moreover:
  \begin{enumerate}
  \item[(1)] The correspondence $\cX$ preserves dimension:
    $\dim \cM = \dim X\cM$.

  \item[(2)] If vectors $f,g$ in $\cH$ are $A$-orthogonal, their images
    $u=Xf$, $v=Xg$ in $\cK$ are $B$-orthogonal.
  \end{enumerate}
  \end{lemma}

\begin{proof}
  Since $X\in\lhk$, $X\in\cL(|\cH|,|\cK|)$ for any associated Hilbert
  spaces.  Thus by standard Hilbert space methods, if $\cM$ is a
  closed subspace of $\cH$, $X\cM$ is a closed subspace of $\cK$ of
  the same dimension.  For any $f,g\in\cH$ and $u=Xf$, $v=Xg$
  in~$\cK$,
\begin{equation}\label{congruent1} {\ip{f}{g}}_A = {\ip{Af}{g}}_\cH =
  {\ip{BXf}{Xg}}_\cK = {\ip{u}{v}}_B.
\end{equation}
It follows that $\cX$ maps $\cC_\pm(\cH,A)$ into $\cC_\pm(\cK,B)$.
Similarly, $\cX$ is onto.  A straightforward verification shows that
$\cX$ is one-to-one, which proves~(1).  Then (2) follows from
\eqref{congruent1}.
\end{proof}

\begin{definition}\label{defhermindices}
  Let $C$ be a selfadjoint operator on a \krein space~$\cH$.  The {\bf
    positive/negative hermitian indices} of $C$, denoted $h_\pm(C)$,
  are defined as the maximum dimensions of a closed $C$-strictly
  positive/negative subspace of~$\cH$.
\end{definition}

Theorem \ref{indices-exist} shows that hermitian indices are well
defined, that is, the required subspaces of maximum dimension always
exist.  The hermitian indices of the identity operator $1$ on a \krein
space $\cH$ are $h_\pm(1) = \indplusminus \cH$, where
$\indplusminus\cH$ are the positive and negative indices of~$\cH$.
The zero operator has hermitian indices $h_\pm(0)=0$.

\begin{theorem}\label{indices-exist} 
  The hermitian indices $h_\pm(C)$ are well defined for every
  selfadjoint operator $C$ on a \krein space~$\cH$ and are invariant
  under congruence.  If $\cH$ is a Hilbert space, then
  $h_\pm(C)=\dim \cH_\pm$, where $\cH_\pm$ are the spectral subspaces
  of $C$ for the sets $(0,\infty)$ and $(-\infty,0)$.
\end{theorem}

\begin{corollary}\label{congruent2sa}
  Every selfadjoint operator on a \krein space is congruent to a
  selfadjoint operator on a Hilbert space having the same hermitian
  indices.
\end{corollary}

\begin{proof}[Proof of Theorem $\ref{indices-exist}$]
  Assume first that $\cH$ is a Hilbert space, and let $E(\cdot)$ be
  the spectral measure for a selfadjoint operator $C$ on~$\cH$.  Set
  \begin{equation*}
    \cH_+ = E((0,\infty)),\quad \cH_- = E((-\infty,0)),
    \quad \cH_0 = E(\{0\}).
  \end{equation*}
  By the spectral theorem, $\cH_\pm$ are $C$-strictly
  positive/negative subspaces of~$\cH$.  Consider any closed
  $C$-strictly positive subspace $\cN$ of~$\cH$, and let
  $T\colon\cN\to\cH_+$ be orthogonal projection onto $\cH_+$.  If
  $f\in\cN$ and $Tf=0$, then $f\in\cH_-\oplus\cH_0$ and hence
  ${\ip{f}{f}}_C \le 0$.  Then $f=0$ because $\cN$ is $C$-strictly
  positive.  Since $T$ is one-to-one, $\dim\cN \le \dim\cH_+$.
  Therefore $\cH_+$ has maximum dimension among all closed
  $C$-strictly positive subspaces of~$\cH$.  The negative case is
  handled similarly.  Hermitian indices are thus well defined for
  selfadjoint operators on a Hilbert space.

  By Lemma \ref{LemmaA}, every selfadjoint operator $C$ on a \krein
  space is congruent to a selfadjoint operator $D$ on a Hilbert space.
  The congruence maps $D$-strictly positive/negative subspaces onto
  $C$-strictly positive/negative subspaces and preserves dimension by
  Lemma \ref{LemmaB}.  Thus hermitian indices are well defined on
  \krein spaces.  Another application of Lemma \ref{LemmaB} shows that
  the indices are invariant under congruence.
\end{proof}

The well-known theorem of J.~J.\ Sylvester that characterizes the
congruence of matrices by their numbers of positive, negative, and
zero eigenvalues has a straightforward generalization to
finite-dimensional \krein
spaces.

\begin{theorem}[Sylvester theorem]\label{sylvesterth}
  Let $A\in\lh$ and $B\in\lk$ be selfadjoint operators on \krein
  spaces $\cH$ and $\cK$ that have the same finite dimension.  Then
  $A$ is congruent to $B$ if and only if $h_\pm(A)=h_\pm(B)$.
\end{theorem}

\begin{proof}
  Assume $h_\pm(A)=h_\pm(B)$.  By Corollary \ref{congruent2sa}, $A$
  and $B$ are congruent to selfadjoint operators $A_1$ and $B_1$ on
  Hilbert spaces $\cH_1$ and $\cK_1$ having the same hermitian indices
  as $A$ and~$B$.  Then $h_\pm(A_1)=h_\pm(B_1)$, and the Hilbert
  spaces $\cH_1$ and $\cK_1$ have the same finite dimension as $\cH$
  and $\cK$.  By Theorem \ref{indices-exist}, $h_\pm(A_1)$ and
  $h_\pm(B_1)$ are the numbers of positive/negative eigenvalues of
  $A_1$ and $B_1$.  Therefore $A_1$ and $B_1$ are congruent by the
  classical theorem of Sylvester \cite[p.\ 223]{HornJohnson1990}, and
  hence $A$ and $B$ are congruent.  The other direction follows from
  Theorem \ref{indices-exist}.
\end{proof}

\goodbreak
\begin{theorem}\label{decomp-exist}
  For every selfadjoint operator $C$ on a \krein space~$\cH$, there
  exist subspaces $\cM_+,\cM_-,\cM_0$ of $\cH$ such that
 \begin{equation}\label{decomp1}
   \cH = \cM_+ + \cM_- + \cM_0,
 \end{equation}
 where $(i)$ $\cM_\pm$ are closed and $C$-strictly positive/negative
 and $\cM_0 = \ker C$; $(ii)$ the sum of any two of the three
 subspaces is closed; $(iii)$ the subspaces are pairwise
 $C$-orthogonal; and $(iv)$ $\dim \cM_\pm = h_\pm(C)$.
 \end{theorem}

 It follows that the hermitian indices $h_\pm(C)$ coincide with the
 ranks of positivity and negativity of $C$ defined in Bogn{\'a}r
 \cite[pp.\ 95, 149]{Bognar1974}.  For by their definition, the latter
 indices are determined by any fundamental decomposition of $\cH$ in
 the $C$-inner product \cite[pp.\ 24,\ 95]{Bognar1974}, and
 \eqref{decomp1} is one such fundamental decomposition.
 \textcolor{red}{ In this context, the term ``fundamental
   decomposition'' is in the sense of Bogn{\'a}r \cite[p.\
   24]{Bognar1974}.}  See also Azizov and Iokhvidov \cite{aziziokh}
 and Gheondea \cite{Gheondea2022} for analogous notions of fundamental
 decompositions and ranks of positivity and negativity.

\begin{proof}[Proof of Theorem $\ref{decomp-exist}$]
  By Lemma \ref{LemmaA}, we can choose a selfadjoint operator $D$ on a
  Hilbert space $\cK$ and an invertible operator $Y\in\lkh$ such that
  $D=Y^*CY$.  Let $\cK_+, \cK_-, \cK_0$ be the spectral subspaces for
  $D$ for $(0,\infty)$, $(-\infty,0)$, $\{0\}$, so $\cK_\pm$ are
  $D$-strictly positive/negative and $\cK_0=\ker D$.  Then
  \eqref{decomp1} holds with
  \begin{equation*} 
      \cM_+ = Y\cK_+, \quad \cM_- = Y\cK_-, \quad \cM_0 = Y\cK_0.
  \end{equation*}
  Since $\cM_0 =Y\ker D = \ker C$, the subspaces $\cM_+,\cM_-,\cM_0$
  satisfy (i), (ii), (iv) by Lemma \ref{LemmaB}(1) because
  $\cK_+, \cK_-, \cK_0$ have these properties in~$\cK$.  The subspaces
  $\cK_+, \cK_-, \cK_0$ are pairwise orthogonal in the inner product
  of $\cK$.  They are also pairwise $D$-orthogonal by the special fact
  that they are invariant under~$D$.  Hence $\cM_+,\cM_-,\cM_0$
  satisfy (iii) by Lemma \ref{LemmaB}(2).
\end{proof}

 \begin{remark}
   Decompositions of the type $\eqref{decomp1}$ in Theorem
   $\ref{decomp-exist}$ are not unique in general.  The proof of
   Theorem $\ref{decomp-exist}$ shows that when $\cH$ is a Hilbert
   space, a decomposition \eqref{decomp1} can be chosen such that the
   subspaces $\cM_+,\cM_-,\cM_0$ are invariant under~$C$.  This fails
   in general for \krein spaces.
\end{remark}

In Theorem \ref{decomp-exist}, the conditions $(i)-(iii)$ alone imply
$(iv$).  This is shown in the next result, which gives additional
information.

\begin{theorem}\label{decomp-converse}
  Let $C$ be a given selfadjoint operator on a \krein space~$\cH$, and
  let $\cM_+,\cM_-,\cM_0$ be any subspaces of $\cH$ that satisfy
  $\eqref{decomp1}$ and  conditions $(i),(ii),(iii)$ in Theorem
  $\ref{decomp-exist}$.  Then the subspaces also satisfy $(iv)$.
  Moreover:
  \begin{enumerate}
  \item The sum $\eqref{decomp1}$ is direct, that is, if
    $f_+ + f_- + f_0 =0$ with $f_\pm\in\cM_\pm$ and $f_0\in\cM_0$,
    then $f_+=f_-=f_0=0$.

  \item There exist projections $Q_\pm,Q_0$ in $\lh$ such that $Q_+$
    has range $\cM_+$ and kernel $\cM_- + \cM_0$; $Q_-$ has range
    $\cM_-$ and kernel $\cM_+ + \cM_0$; $Q_0$ has range $\cM_0$ and
    kernel $\cM_+ + \cM_-$.
  \end{enumerate}
\end{theorem}

\begin{proof}
  Let $\cM_+,\cM_-,\cM_0$ satisfy $\eqref{decomp1}$ and conditions
  $(i),(ii),(iii)$ in Theorem $\ref{decomp-exist}$.  Suppose
  $f_\pm\in\cM_\pm$, $f_0\in\cM_0$ and $f_+ + f_- + f_0 = 0$.  By $(i)$,
  ${\ip{f_+}{f_+}}_C \ge 0$ with equality only when $f_+=0$;
  ${\ip{f_-}{f_-}}_C \le 0$ with equality only when $f_-=0$; and
  ${\ip{f_0}{f_0}}_C= {\ip{Cf_0}{f_0}}_\cH=0$ because $Cf_0=0$.  Since
  $\cM_- \perp_C \cM_0$ by $(iii)$,
  \begin{equation*}
      {\ip{f_+}{f_+}}_C = {\ip{-f_- - f_0}{-f_- - f_0}}_C
      = {\ip{f_-}{f_-}}_C + {\ip{f_0}{f_0}}_C
      \le 0.
  \end{equation*}
  Therefore ${\ip{f_+}{f_+}}_C = 0$, and hence $f_+=0$.  Then also
  $f_-+f_0=0$, so
  \begin{equation*}
      {\ip{f_-}{f_-}}_C = {\ip{-f_0}{-f_0}}_C = 0,
  \end{equation*}
  and hence $f_-=0$.  So $f_+ = f_- = f_0 = 0$, which proves~(1).

  By $(ii)$ and a closed graph argument, there are projection operators
  $Q_\pm, Q_0$ in $\lh$ such that
   \begin{equation*}
      Q_+f = f_+, \qquad Q_-f = f_-, \qquad Q_0f = f_0,
    \end{equation*}
    whenever $f = f_+ + f_- + f_0$ with $f_\pm\in\cM_\pm$ and
    $f_0\in\cM_0$.  They have the properties listed in (2) by
    construction.

    It remains to show that $\cM_+,\cM_-$ satisfy $(iv)$, that is,
    $\dim\cM_\pm = h_\pm(C)$, or, equivalently, $\cM_\pm$ have maximum
    dimensions among all closed $C$-strictly positive/negative
    subspaces.  Consider any closed $C$-strictly positive subspace
    $\cN$ of $\cH$, and define $T\colon\cN\to\cM_+$ by
    $Tf = Q_+f$, $f\in\cN$.
      Then $T$ is continuous in the strong topology of~$\cH$.  It is
    also one-to-one.  For suppose $f\in\cN$ and $Tf=0$.  Then
    $Q_+f=0$, and hence $f=f_-+f_0$, where $f_-\in\cM_-$ and
    $f_0\in\cM_0$.  Since $\cM_-\perp_C\cM_0$,
    \begin{equation*}
      {\ip{f}{f}}_C ={\ip{f_-}{f_-}}_C + {\ip{f_0}{f_0}}_C \le 0.
    \end{equation*}
    Therefore ${\ip{f}{f}}_C=0$ and hence $f=0$.  It follows that
    $\dim \cN \le \dim \cM_+$, which proves $(iv)$ for the positive
    case. The negative case is proved similarly, and the result
    follows.
  \end{proof}
  
  A {\bf signature operator} on a Hilbert space $\cH$ is an operator
  $J\in\lh$ which is both selfadjoint and unitary.  If $J$ is a
  signature operator on~$\cH$, it coincides with $\pm1$ on orthogonal
  closed subspaces $\cM_\pm$ that span $\cH$, and
  $h_\pm(J)= \dim \cM_\pm$.
 
   \begin{theorem}\label{keyth}
    Let $\cH$ be a Hilbert space, $C\in\lh$ a selfadjoint operator
    such that $\ker C = \{0\}$.  Suppose $C=T^*J_A T$, where $J_A$ is a
    signature operator on a Hilbert space~$\cK$, and $T\in\lhk$ has
    zero kernel and dense range in~$\cK$.  Then
    $h_\pm(C)= h_\pm(J_A)$.
  \end{theorem}

  The assumption that $T$ has zero kernel is redundant since
  $\ker C = \{0\}$ but retained for convenience in the proof.  We give
  a \krein space proof based on a theorem of R.~S. Phillips
  \cite{Phillips1961} in the form given in Arsene and Gheondea
  \cite{AG1982}.

  \begin{proof}
    Let $\cA$ be the \krein space which is $\cK$ as a vector space and
    has $J_A$ as fundamental symmetry.  Then
    ${\ip{f}{g}}_{\cA} = {\ip{J_Af}{g}}_{\cK}$, $f,g\in\cA$.
 Suppose $J_A=\pm1$ on the subspaces $\cA_\pm$ of $\cK$.  Then
 $h_\pm(J_A)=\dim\cA_\pm$ and $\cA$ has fundamental decomposition
 $\cA=\cA_+\oplus_\cA\cA_-$.  The corresponding associated Hilbert
 space $|\cA|$ coincides with the Hilbert space~$\cK$.  In particular,
 $\cA$ and $\cK$ have the same strong topology.

 Since $\ker C= \{0\}$, we can write $\cH=\cH_+\oplus_\cH\cH_-$, where
 $\cH_\pm$ are the spectral subspaces for $C$ for the positive and
 negative real axes, and then $h_\pm(C)=\dim\cH_\pm$ by Theorem
 \ref{indices-exist}.  Let $J_C$ be the signature operator on $\cH$
 such that $J_C=\pm1$ on $\cH_\pm$, and let $|C| = (C^2)^\frac12$ be
 the operator modulus of~$C$.  Then
 \begin{equation}\label{keyfact1}
 |C|^\frac12 J_C  |C|^{\frac12} = C = T^*J_AT
 \end{equation}
and 
 \begin{equation}\label{keyfact2}
 |C|^\frac12 \cH_\pm\subseteq\cH_\pm.
 \end{equation}
 The problem is to show that $\dim \cH_\pm = \dim\cA_\pm$.

 Define subspaces $\cG$ and $\cG_\pm$ of $\cA$ by $\cG=T\cH$ and
 $\cG_\pm=T\cH_\pm$.  Then $\cG = \cG_+ + \cG_-$.  Since $T$ has dense
 range in $\cK$, $\cG$ is dense in~$\cA$ because $\cA$ and $\cK$ have
 the same strong topology.  The subspaces $\cG_\pm$ are
 nonnegative/nonpositive and orthogonal in~$\cA$.  For suppose
 $g_\pm\in\cG_\pm$ with $g_\pm = Tf_\pm$, $f_\pm\in\cH_\pm$, then by
 \eqref{keyfact1} and \eqref{keyfact2},
 \begin{equation*} {\ip{g_\pm}{g_\pm}}_\cA = {\ip{J_A
       Tf_\pm}{Tf_\pm}}_\cK = {\ip{J_C |C|^\frac12 f_\pm}{|C|^\frac12
       f_\pm}}_\cH,
  \end{equation*}
  where the last term on the right side is nonnegative/nonpositive.
  Similarly,
  \begin{equation*} {\ip{g_+}{g_-}}_\cA = {\ip{J_A Tf_+}{Tf_-}}_\cK =
    {\ip{J_C |C|^\frac12f _+}{|C|^\frac12 f_-}}_\cH = 0,
  \end{equation*}
 which shows orthogonality.

  Let 
  \begin{equation}\label{mar12a2024}
    \cG_+ = \{x + G_+ x\}_{x\in\cM_+},
    \qquad
    \cG_- = \{G_- y + y\}_{y\in \cM_-},
  \end{equation}
  be the graph representations of $\cG_\pm$ relative to the
  fundamental decomposition $\cA = \cA_+\oplus_\cA \cA_-$ of~$\cA$.
  Thus $G_+\colon \cM_+\to|\cA_-|$ is a contraction from a subspace
  $\cM_+$ of $\cA_+$ into $|\cA_-|$, and $G_-\colon \cM_-\to\cA_+$ is
  a contraction from a subspace $\cM_-$ of $|\cA_-|$ into $\cA_+$.
  Let $P_\pm\in\cL(\cA)$ be the selfadjoint projections onto
  $\cA_\pm$.  The subspace $\cM_+$ in \eqref{mar12a2024} is the
  one-to-one image of $\cH_+$ under $P_+T$ ($T$ is one-to-one by
  assumption, and the restriction of $P_+$ to $\cG_+$ is one-to-one by
  the graph representation of $\cG_+$), and hence
  \begin{equation}\label{dimA}
    \dim \cH_+ = \dim {\overline{\cM}}_+ \le \dim\cA_+.
  \end{equation}
  Similarly, $\cM_-$ is the one-to-one image of $\cH_-$ under $P_-T$,
  and hence
  \begin{equation}\label{dimB}
    \dim \cH_- = \dim {\overline{\cM}}_- \le \dim\cA_-.
  \end{equation}
  
  Since $\cG_+$ is nonnegative, $\cG_-$ nonpositive, and
  $\cG_+ \perp_\cA \cG_-$, by Phillips' Theorem \cite{Phillips1961,
    AG1982}, there is a maximal nonnegative subspace $\widetilde\cG_+$
  and a maximal nonpositive subspace $\widetilde\cG_-$ such that
  $ \cG_\pm\subseteq\widetilde\cG_\pm$ and
  $\widetilde\cG_+ \perp_\cA \widetilde\cG_-$.  Then
  \begin{equation*}
    \widetilde\cG_+ = \{x + G x\}_{x\in\cA_+},
    \qquad
    \widetilde\cG_- = \{G^* y + y\}_{y\in |\cA_-|},
  \end{equation*}
  where $G\in\cL(\cA_+,|\cA_-|)$ is a contraction operator.  
  The inclusions $\cG_\pm\subseteq\widetilde\cG_\pm$ mean that
  $G_+ = G\big\vert\cM_+$ and $G_- = G^*\big\vert\cM_-$, and hence
  \eqref{mar12a2024} takes the form
  \begin{equation}\label{mar12b2024}
    \cG_+ = \{x + G x\}_{x\in\cM_+},
    \qquad
    \cG_- = \{G^* y + y\}_{y\in \cM_-},
  \end{equation}

  To complete the proof, consider any $h\in\cA$.  Since
  $\cG = \cG_++\cG_-$ is dense in $\cA$, by \eqref{mar12b2024} there
  exist sequences $\{f_n^\pm\}_{n=1}^\infty\subseteq\cM_\pm$ such that
  \begin{equation}\label{mar12c2024}
    \lim_{n\to\infty} \big[ f_n^+ + Gf_n^+ + G^*f_n^- + f_n^- 
    \big] = h.
  \end{equation}
  If $h\in\cA_+$, we can apply $P_\pm$ to both sides of
  \eqref{mar12c2024} to get $f_n^++G^*f_n^-\to h$ and
  $Gf_n^+ + f_n^-\to 0$, and hence
  \begin{equation*}
    (1-G^*G)f_n^+ = f_n^++G^*f_n^- - G^*(Gf_n^+ + f_n^-) \to h.
  \end{equation*}
  It follows that $1-G^*G$ maps $\overline{\cM}_+$ onto a dense
  subspace of $\cA_+$.  Therefore
   \begin{equation*}
    \dim \cA_+ \le \dim \overline{\cM}_+ = \dim\cH_+.
  \end{equation*}
  The reverse inequality was proved above, so
  $\dim \cA_+ = \dim\cH_+$.  If $h\in\cA_-$, then from
  \eqref{mar12c2024} we get $f_n^+ + G^*f_n^-\to 0$ and
  $Gf_n^+ + f_n^-\to h$, and then
  \begin{equation*}
    (1-GG^*)f_n^- =
    Gf_n^+ + f_n^- - G(f_n^+ + G^*f_n^-) \to h.
  \end{equation*}
  Therefore $1-GG^*$ maps
  $\overline{\cM}_-$ onto a dense subspace of $\cA_-$, and
  \begin{equation*}
    \dim \cA_- \le \dim \overline{\cM}_- = \dim \cH_-.
  \end{equation*}
  Thus $\dim \cA_- = \dim\cH_-$.  We have shown that
  $\dim \cA_\pm = \dim\cH_\pm$, and the result follows.
\end{proof}

\section{Bogn{\'a}r-Kr{\'a}mli factorization }\label{Sect3}

\textcolor{black}{Following terminology introduced in \cite{DRFields},
  by a {\bf Bogn{\'a}r-Kr{\'a}mli factorization} of a selfadjoint
  operator $C$ on a \krein space $\cH$ we mean any representation
  $C=AA^*$ such that $A\in\lah$ for some \krein space $\cA$ and
  $\ker A = \{0\}$.  In the viewpoint of Constantinescu and Gheondea
  \cite{CG2} and Gheondea \cite[Ch.\ 6]{Gheondea2022}, $\cA$ is an
  {\bf induced \krein space}.  Another equivalent notion is that of
  {\bf continuously contained \krein spaces}, which are discussed
  below.  }

Bogn{\'a}r-Kr{\'a}mli factorizations of a selfadjoint operator $C$ are
far from unique.  In a special case, any two are identical modulo an
isomorphism between the external domain spaces, and then the
factorization is said to be {\bf essentially unique}.  This occurs if
and only if $C$ is congruent to a selfadjoint operator on a Hilbert
space whose spectrum omits an interval $(0,\varepsilon)$ or
$(-\varepsilon,0)$ for some $\varepsilon>0$.  For details and other
conditions, see Hara \cite{Hara1992}, Dritschel \cite{MAD1993},
Constantinescu and Gheondea \cite{CG1, CG2}, and Gheondea \cite[Ch.\
6]{Gheondea2022}.  The indices of the \krein space $\cA$ in any
Bogn{\'a}r-Kr{\'a}mli factorization are nevertheless unique and
determined by the selfadjoint operator $C$.

\begin{theorem}\label{mainth}
  A selfadjoint operator $C$ on a \krein space $\cH$ admits a
  Bogn{\'a}r-Kr{\'a}mli factorization 
  \begin{equation}\label{bk-1}
  C=AA^*, \quad A\in\lah, \quad \ker A = \{0\}
  \end{equation}
  for some \krein space $\cA$ if and only if\/
  $\indplusminus \cA = h_\pm(C)$.
\end{theorem}

\textcolor{black}{ By Theorem \ref{mainth}, the standing assumption of
  separability in \cite[p.\ 144]{DRFields} can be removed.}

 \textcolor{red}{The sufficiency part in Theorem \ref{mainth} is a
   known variant of \cite[Th.\ 1]{BognarKramli}.}
 Proofs can be found in \cite{CG1, DR1990, DRFields, Gheondea2022},
 and a proof is included here for completeness.

 \textcolor{red}{A proof of necessity is given in \cite{DR1990} for
   separable \krein spaces.  The proof presented here is new and based
   on Theorem \ref{keyth}.  Its key feature is that it does not
   require separability.}
    
    \begin{proof}[Proof of Theorem $\ref{mainth}$]
   Necessity.  Assume given a \krein space $\cA$ and factorization
   $C=AA^*$ with $A\in\lah$ and $\ker A = \{0\}$.  We must show that
   $\indplusminus \cA = h_\pm(C)$.  Without loss of generality, we can
   assume that $\cH$ is a Hilbert space.  For suppose the conclusion
   holds in this case.  For the general case when $\cH$ is a \krein
   space, by Corollary \ref{congruent2sa}, $C$ is congruent to a
   selfadjoint operator $D$ on a Hilbert space~$\cK$ such that
   $h_\pm(D) = h_\pm(C)$.  Suppose $C=X^*DX$, where $X\in\lhk$ is
   invertible.  Since $X^{*\,-1} = X^{-1\,*}$,
  \begin{equation*}
    D = X^{*\ -1}CX^{-1} = X^{-1\,*}AA^*X^{-1} = BB^*,
  \end{equation*}
  where $B=X^{-1\,*}A\in\lak$ and $\ker B = \{0\}$.  Since $\cK$ is a
  Hilbert space, $h_\pm(D) = \indplusminus \cA$ by the special case,
  and hence $h_\pm(C) = \indplusminus \cA$.

  Thus we can assume that $\cH$ is a
  Hilbert space.  We can further assume that $\ker C = \{0\}$.  For
  suppose the conclusion is known in this case.  Consider a
  selfadjoint operator $C$ on a Hilbert space $\cH$ with possibly
  $\ker C \neq\{0\}$.  Let $\cB$ be a \krein space, $B\in\lbh$,
  $\ker B = \{0\}$, and $C=BB^*$.  By the spectral theorem, we can
  write $\cH=\cH_1\oplus\cH_0$ and $\cH_1 = \cH_+\oplus\cH_-$, where
  $\cH_\pm$ and $\cH_0 = \ker C$ are the spectral subspaces for $C$
  for $(0,\infty)$, $(-\infty,0)$, and $\{0\}$.  Notice that
  \begin{equation*}
    B(B^*\cH) = C\cH \subseteq \cH_1,
  \end{equation*}
  that is, $B$ maps the dense subspace $B^*\cH$ of $\cB$ into $\cH_1$.
  Therefore $\ran B\subseteq\cH_1$.  Set $C_1 = E^*CE = B_1B_1^*$,
  where $E\in\cL(\cH_1,\cH)$ is the natural embedding of $\cH_1$ into
  $\cH$, and $B_1 = E^*B\in\cL(\cB,\cH_1)$.  Then $E^*$ is orthogonal
  projection from $\cH$ onto $\cH_1$, and $\ker B_1=\{0\}$ because
  $\ran B\subseteq\cH_1$ and $E^*$ is the identity on $\cH_1$.
  Moreover, $C_1=C\big\vert\cH_1$ and hence $\ker C_1 = \{0\}$.  Since
  we grant the result in the zero kernel case,
  $\indplusminus\cB = h_\pm(C_1)$.  Then noting that
  $h_\pm(C_1) = \dim \cH_\pm = h_\pm(C)$, we obtain
  $\indplusminus \cB = h_\pm(C)$, as was to be shown.

  For the remaining case, let $\cH$ be a Hilbert space, and assume
  $\ker C = \{0\}$.  Let $\cA$ be a \krein space and $A\in\lah$ an
  operator such that, $\ker \cA = \{0\}$, and $C=AA^*$.  Choose a
  fundamental symmetry $J_\cA$ for $\cA$ and corresponding associated
  Hilbert space $|\cA|$.  Then $C=AJ_A A^\times$, where
  $A^\times\in\cL(\cH,|\cA|)$ is the adjoint of $A$ viewed as an
  operator in $\cL(|\cA|,\cH)$.  Since $\ker A = \{0\}$, $A^\times$
  has dense range.  Since $\ker C = \{0\}$, $A^\times$ has zero
  kernel.  Hence by Theorem \ref{keyth} applied with
  $T = A^\times \in\cL(\cH,|\cA|)$, we obtain
  $h_\pm(C) = h_\pm(J_\cA) = \indplusminus \cA$, as was to be shown.

   Sufficiency.  Suppose first that $\cH$ is a Hilbert space and $\cA$
   is a \krein space such that $\indplusminus \cA = h_\pm(C)$.  By
   Theorem \ref{keyfact2}, $h_\pm(C) = \dim \cH_\pm$, where $\cH_\pm$
   are the spectral subspaces for $C$ for the intervals $(0,\infty)$
   and $(-\infty,0)$.  Since $\cA$ can be replaced by an isomorphic
   copy and $h_\pm(C) = \dim \cH_\pm$, we are free to replace $\cA$ by
   $\cA = \cH_+ + \cH_-$ with fundamental symmetry
   $J_\cA(x_+ + x_-) = x_+ - x_-$, $x_\pm\in\cH_\pm$, and inner
   product
\begin{equation*}
 {\ip{x}{y}}_\cA =  {\ip{J_\cA x}{y}}_\cH , \qquad x,y\in\cA.  
\end{equation*}
Define $A\in\lah$ by $Ax = |C|^\frac12 x$ for all $x\in\cA$.  Then
$\ker A = \{0\}$ and $AA^*h = |C|^\frac12 J_\cA |C|^\frac12 h = Ch$
for all $h\in\cH$.  Thus $C=AA^*$, where $A\in\lah$ and
$\ker A = \{0\}$, as required.  The general case of sufficiency, in
which $\cH$ is a \krein space, can be reduced to the Hilbert space
case using Corollary \ref{congruent2sa}, in a similar manner to an
argument given above.
\end{proof}

   A \krein space $\cA$ is {\bf continuously contained} in a \krein
   space $\cH$ if $\cA$ is a linear subspace of $\cH$ and the
   inclusion mapping $A$ from $\cA$ into $\cH$ is continuous.  

   \begin{theorem}\label{mainth2}
     If $\cA$ is a \krein space which is continuously contained in a 
     \krein space $\cH$ with inclusion mapping~$A$, then 
     $\indplusminus \cA = h_\pm(C)$, where
     $C=AA^*$.  
     Every selfadjoint operator $C$ on $\cH$ arises in this way.
    \end{theorem}    
   
   \begin{proof}
   The first statement is a special case of Theorem \ref{mainth}.  For the 
   second, first factor the given selfadjoint operator $C$ in the form
   $C=AA^*$, where  $A\in\lah$ for some \krein
     space $\cA$ and $\ker A = \{0\}$,  Then replace $\cA$ by 
   the range of $A$ in the inner product that makes $A$
     an isomorphism.  
   \end{proof}
   
   Heinz Langer took an early interest in continuously contained
   \krein spaces and de~Branges' theory \cite{deB1988} of
   complementation of contractively contained spaces (1988 private
   communication to Dritschel and Rovnyak).  He pursued this topic to
   great depth in joint work with Branko {\'C}urgus \cite{CL1990,
     CL2003}, showing novel interactions between de~Branges' theory,
   definitizable selfadjoint operators, and the problem of \krein
   space completions of nondegenerate inner product spaces.  Details
   are beyond the present scope, but a re-examination of these topics
   could well shed further light on this area.
      
\medskip
\textcolor{black}{
  \noindent {\bf Funding:} Alejandra Maestripieri was partially
  supported by the Air Force Office of Scientific Research (USA) grant
  FA9550-24-1-0433.}

\def\cprime{$'$} \def\cprime{$'$}
\providecommand{\bysame}{\leavevmode\hbox to3em{\hrulefill}\thinspace}
\providecommand{\MR}{\relax\ifhmode\unskip\space\fi MR }
\providecommand{\MRhref}[2]{%
  \href{http://www.ams.org/mathscinet-getitem?mr=#1}{#2}
}
\providecommand{\href}[2]{#2}

        

\begin{thebibliography}{10}

\bibitem{AG1982}
Gr. Arsene and A.~Gheondea, \emph{Completing matrix contractions}, J. Operator
  Theory \textbf{7} (1982), no.~1, 179--189. \MR{650203}

\bibitem{aziziokh}
T.~Ya. Azizov and I.~S. Iokhvidov, \emph{Linear operators in spaces with an
  indefinite metric}, Pure and Applied Mathematics (New York), John Wiley \&
  Sons, Ltd., Chichester, 1989, Translated from the Russian by E. R. Dawson
  (``Nauka'', Moscow, 1986; MR0863885), A Wiley-Interscience Publication.
  \MR{1033489}

\bibitem{Bognar1974}
J.~Bogn\'ar, \emph{Indefinite inner product spaces}, Ergebnisse der Mathematik
  und ihrer Grenzgebiete, vol. Band 78, Springer-Verlag, New York-Heidelberg,
  1974. \MR{467261}

\bibitem{BognarKramli}
J.~Bogn{\'a}r and A.~Kr{\'a}mli, \emph{Operators of the form {$C^{\ast} C$} in
  indefinite inner product spaces}, Acta Sci. Math. (Szeged) \textbf{29}
  (1968), 19--29. \MR{0240658}

\bibitem{CG1}
T.~Constantinescu and A.~Gheondea, \emph{Elementary rotations of linear
  operators in {K}re{\u\i}n spaces}, J. Operator Theory \textbf{29} (1993),
  no.~1, 167--203. \MR{1277972}

\bibitem{CG2}
\bysame, \emph{Representations of {H}ermitian kernels by means of {K}re\u\i n
  spaces}, Publ. Res. Inst. Math. Sci. \textbf{33} (1997), no.~6, 917--951.
  \MR{1614572}

\bibitem{CL1990}
B.~{\'C}urgus and H.~Langer, \emph{On a theorem of de~{B}ranges}, 1990.

\bibitem{CL2003}
\bysame, \emph{Continuous embeddings, completions and complementation in
  {K}rein spaces}, Rad. Mat. \textbf{12} (2003), no.~1, 37--79. \MR{2022245}

\bibitem{deB1988}
L.~de~Branges, \emph{Complementation in {K}re\u in spaces}, Trans. Amer. Math.
  Soc. \textbf{305} (1988), no.~1, 277--291. \MR{920159}

\bibitem{MAD1993}
M.~A. Dritschel, \emph{The essential uniqueness property for operators on
  {K}re\u in spaces}, J. Funct. Anal. \textbf{118} (1993), no.~1, 198--248.
  \MR{1245603}

\bibitem{DR1990}
M.~A. Dritschel and J.~Rovnyak, \emph{Extension theorems for contraction
  operators on {K}re\u\i n spaces}, Extension and interpolation of linear
  operators and matrix functions, Oper. Theory Adv. Appl., vol.~47,
  Birkh\"auser, Basel, 1990, pp.~221--305. \MR{1120277}

\bibitem{DRFields}
\bysame, \emph{Operators on indefinite inner product spaces}, Lectures on
  operator theory and its applications ({W}aterloo, {ON}, 1994), Fields Inst.
  Monogr., vol.~3, Amer. Math. Soc., Providence, RI, 1996, pp.~141--232.
  \MR{1364446}

\bibitem{Gheondea2022}
A.~Gheondea, \emph{An indefinite excursion in operator theory---geometric and
  spectral treks in {K}re\u\i n spaces}, London Mathematical Society Lecture
  Note Series, vol. 476, Cambridge University Press, Cambridge, 2022.
  \MR{4439693}

\bibitem{Hara1992}
T.~Hara, \emph{Operator inequalities and construction of {K}re\u\i n spaces},
  Integral Equations Operator Theory \textbf{15} (1992), no.~4, 551--567.
  \MR{1166707}

\bibitem{HornJohnson1990}
R.~A. Horn and C.~R. Johnson, \emph{Matrix analysis}, Cambridge University
  Press, Cambridge, 1990, Corrected reprint of the 1985 original. \MR{1084815}

\bibitem{Phillips1961}
R.~S. Phillips, \emph{The extension of dual subspaces invariant under an
  algebra}, Proc. {I}nternat. {S}ympos. {L}inear {S}paces ({J}erusalem, 1960),
  Jerusalem Academic Press, Jerusalem, 1961, pp.~366--398. \MR{133686}

\bibitem{Potapov1955}
V.~P. Potapov, \emph{The multiplicative structure of {$J$}-contractive matrix
  functions}, Trudy Moskov. Mat. Ob\v s\v c. \textbf{4} (1955), 125--236,
  [Russian]; English Translation: AMS Translations, Series 2, 85 (1969),
  115-143. \MR{76882}

\end{thebibliography}
 
\end{document}